\documentclass[12pt]{amsart}
\usepackage{amsmath,amssymb}
\usepackage{graphicx,bm}
\usepackage[margin=1in]{geometry}
\usepackage{hyperref}
\usepackage{amsthm}
\usepackage{verbatim}
\usepackage{xcolor}
\usepackage{enumitem}
\numberwithin{equation}{section}

\newtheorem{thm}{Theorem}[section]
\newenvironment{thmv}[1]{\renewcommand\thethm{\ref{#1}}\begin{thm}%
}{\end{thm}}

\newtheorem{cor}[thm]{Corollary}

\newtheorem{defn}[thm]{Definition}

\begin{document}

\title[Uniqueness for Hypersurfaces in Hyperbolic Space]{Uniqueness of Hypersurfaces of  Constant Higher Order Mean Curvature  in Hyperbolic Space}
\author{Barbara Nelli, Jingyong Zhu}
\thanks{}
\thanks{}

\address{}
\address{Barbara Nelli: DISIM, Universit\'a degli Studi dell'Aquila, Italy}
\email{barbara.nelli@univaq.it}
\address{Jingyong Zhu: Max Planck Institute for Mathematics in the Sciences, Inselstrasse 22, 04103 Leipzig, Germany}
\email{jizhu@mis.mpg.de}

\subjclass[2010]{}%


\begin{abstract}
We study the uniqueness of horospheres  and equidistant spheres in hyperbolic space under different conditions.  First we generalize the Bernstein theorem by  Do Carmo and Lawson \cite{do1983alexandrov} to the embedded hypersurfaces with constant higher order mean curvature. Then we prove two Bernstein type results for immersed hypersurfaces under different assumptions.
 Last, we show the rigidity of horospheres and equidistant spheres in terms of their higher order mean curvatures.
\end{abstract}
\maketitle


\section{Introduction}

In 1927,  S. N. Bernstein proved that the only entire minimal graphs in $\mathbb{R}^3$ are planes. The analogous problem  in higher dimension is known as  Bernstein problem. Namely, given $u: \mathbb{R}^n\to \mathbb{R}$ a minimal graph, is the graph of $u$ a flat hyperplane?  It turns out that the answer is yes for $n\leq7$ and that  has been settled down by a series of very significative papers \cite{almgren1966some,de1965estensione,fleming1962oriented,simons1968minimal}.

On the contrary, for $n\geq8$, one has the famous counterexamples by E. Bombieri, E. De Giorgi and  E. Giusti \cite{BDG}. Afterwards, many generalizations of Bernstein theorem have arised. As an example, we mention 
\cite{schoen1975curvature}
where  R. Schoen, L. Simon and  S. T. Yau studied a Bernstein type theorem for  stable minimal hypersurfaces. 

Later on,  Bernstein type  theorems  for constant mean curvature  hypersurfaces in Euclidean  and  in hyperbolic space 
$\mathbb{H}^{n+1}$ have been studied.   Let  us give a very simple example in the Euclidean space. 
Does it exist an entire graph $M$  in $\mathbb{R}^{n+1}$ with constant mean curvature $H\not=0$?
It is well known that  the answer is no and here is the proof. Assume $M$ exists. Without loss of generality, we can assume that the mean curvature vector of $M$ points upward. Consider a sphere $\mathcal S$  in $\mathbb{R}^{n+1}$ of mean curvature $H$.
As $\mathcal S$ is compact, up to an ambient isometry, we can assume that $\mathcal S$ is above $M.$ Then, translate down 
$\mathcal S.$ Clearly, there will be a first contact point  $p$ between $M$ and $\mathcal S.$  At $p,$  $\mathcal S$  and $M$ are tangent and  applying the maximum principle,  one gets that $\mathcal S$ and $M$ should coincide, that is a contradiction.

 In the hyperbolic space, there is more variety  of constant mean curvature hypersurfaces.
We point out three important results   that have been proved throughout history.

{\em

A complete hypersurface $\Sigma$ of  ${\mathbb H}^{n+1}$ with constant mean curvature   is a horosphere, provided:

\begin{enumerate}
\item  \label{DL} $\Sigma$ is properly embedded and has  exactly one point in its asymptotic boundary \cite[Theorem A]{do1983alexandrov}.
\item \label{AD} $n=2$ and $\Sigma$ is properly immersed between two horospheres  in ${\mathbb H}^3$ with the same asymptotic point \cite[Theorem 1]{alias2006uniqueness}.
\item \label{BQJ} $\Sigma$ is immersed,  has all the principal curvatures uniformly are larger than $-1$  and has  exactly one point in its asymptotic boundary \cite[Theorem 1.5]{bonini2017weakly}.
\end{enumerate}
}

We will study analogous problems for hypersurfaces with constant higher order  mean curvature functions 
 ($H_r$-hypersurface in the following). Our results are generalizations of the three statements above.
  Moreover, motivated by the recent work by R. Souam \cite{souam2019mean}, we show the $r$-mean curvature rigidity of horospheres and equidistant spheres (notice that $H_r$ may be zero). 
  
  The article is organized as follows. In Section 2, we fix notations and collect some preliminary results. 
 The result of  Section 3 is  the following  uniqueness theorem for  horospheres and equidistant spheres, which is a generalization of \eqref{DL}.
  
\begin{thmv}{DL r-mean}
 Let  $\Sigma$ be a complete $H_r$-hypersurface properly embedded in the hyperbolic space $\mathbb{H}^{n+1},$ $r\geq 2.$ Denote the asymptotic boundary of $\Sigma$ by $\partial_\infty\Sigma.$ Then we have:

(1) if $\partial_\infty\Sigma$ is a point, then $\Sigma$ is a horosphere;

(2) if $\partial_\infty\Sigma$ is a sphere and $\Sigma$ separates poles, then $\Sigma$ is a equidistant sphere.
\end{thmv}

 For the definition of hypersurface separating poles, see Section  3.
For $r=1$ the result is contained in \cite[Theorem B]{do1983alexandrov}.
 
In Section  4,  we consider the Bernstein problem for immersed hypersurfaces, either with constant  $r$-mean curvature 
 or satisfying a Weingarten  equation. In particular, we generalize  results by  L. Al{\'\i}as and  M. Dajczer,  \cite{alias2006uniqueness}, which concerns a more general problem in  warped products, by L.  Al{\'\i}as, D.  Impera and M. Rigoli \cite{alias2013hypersurfaces} and  by  Bonini, Qing and the second author \cite{bonini2017weakly}. Let us mention our main results in Section 4.

\begin{thmv}{theorem-parabolic}
Let $\Sigma$  be a $r$-admissible, $L_{r-1}$-parabolic  $H_r$-hypersurface properly immersed in ${\mathbb H}^{n+1}$. If  $\Sigma$  is contained in a slab  and  the angle function does not change sign, then $\Sigma$ is a horosphere.
\end{thmv}

\begin{thmv}{weakly horospherically convex}
Let $\Sigma$ be an immersed, complete, uniformly admissible Weingarten hypersurface  in $\mathbb{H}^{n+1}.$  Then $\Sigma$  is a horosphere provided its asymptotic boundary is a single point.
\end{thmv}

A slab is the space between two horospheres that share the same asymptotic point. For the definition of  $r$-admissibility, $L_{r-1}$-parabolicity,  uniformly admissible Weingarten hypersurface, see Section 2.

Finally, in Section 5,  we show the $r$-mean curvature rigidity of horospheres and equidistant spheres, which generalizes  \cite{souam2019mean}.

\begin{thmv}{rigidity}
Let $M$ be a horosphere or a equidistant sphere in hyperbolic space $\mathbb{H}^{n+1}$, $n\geq2$ and $H_M>0$ denote its $r$-mean curvature, $r\geq1$, with respect to the orientation given by the mean curvature vector. Let $\Sigma$ be a connected properly embedded  $C^2$ hypersurface in $\mathbb{H}^{n+1}$ which coincides with $M$ outside a compact subset $B$ in $\mathbb{H}^{n+1}$. Choose the orientation on $\Sigma$ such that the $r$-mean curvature $H_r$ of $\Sigma$ is equal to $H_M$ outside the compact set $B$. With respect to this orientation, if either $H_r\geq H_M$ or $|H_r|\leq H_M$, then $\Sigma=M$.
\end{thmv}

{\bf Acknowledgements}
The first author  was partially supported by  INdAM-GNSAGA.
Part of the present work was done during the visit of the second author to Scuola Normale Superiore. The second author would like to thank Professor J\"urgen Jost and Professor Andrea Malchiodi for their support to his visit.

\section{Preliminaries}

\subsection{Models for the hyperbolic space ${\mathbb H}^{n+1}$}
 \label{models}
 We will work in different well-known models for the hyperbolic space. For the sake of completeness we briefly describe them.
 
 \ 
 
 {\em The half-space model.}  Consider the  upper half-space 
  \begin{align*}
   \mathbb{R}^{n+1}_+=\{(x_1,\cdots,x_{n+1})\in\mathbb{R}^{n+1}|x_{n+1}>0\}
   \end{align*}
    with the metric $\frac{dx_1^2+\dots+dx_{n+1}^2}{x_{n+1}^2}.$ In this model, horospheres are either horizontal hyperplanes
    or Euclidean spheres tangent at  some point to  the hyperplane $\{x_{n+1}=0\}$. Moreover,  the intersections of the upper half-space with Euclidean spheres not contained in the upper half-space are totally umbilical hypersurfaces, whose absolute values of the principal curvatures are strictly less than $1$.  Such hypersurfaces are usually called {\em equidistant spheres} when the principal curvatures are not zero. The ones with centers on the hyperplane $\{x_{n+1}=0\}$  are {\em (totally geodesic) hyperplanes}. 
\

{\em The warped product model.} $\mathbb{H}^{n+1}$ can be viewed as the warped product $\mathbb{R}\times_{e^t}\mathbb{R}^n$, that  is the product manifold ${\mathbb R}\times\mathbb{R}^n$ endowed with the following metric
	\begin{equation}
    \langle\cdot,\cdot\rangle=\pi_1^*(dt^2)+e^{2t}\pi_{2}^*(\langle\cdot,\cdot\rangle_{\mathbb{R}^{n}}),
    \end{equation}
    where $\pi_1$ and $\pi_2$ denote the projections onto the two factors  and $\langle\cdot,\cdot\rangle_{\mathbb{R}^{n}}$ is the Euclidean metric. Notice that the  leafs $\mathbb{R}_t=\{t\}\times\mathbb{R}^n$ are  horospheres with r-mean curvature one, for every 
    $r=1,\dots,n$   with respect to $-T$, where $T$ is the lift of $\frac{\partial}{\partial t}$. All the $\mathbb{R}_t$'s share the same point at infinity.
    For an immersed hypersurface $\Sigma$ of $\mathbb{R}\times_{e^t}\mathbb{R}^n$, oriented by $\nu$, we define the height function $h\in C^\infty(\Sigma)$ to be the restriction of $\pi_{1}$ to $\Sigma$ and  the angle function by $\Theta=\langle\nu,T\rangle$.

    \ 
    
{\em The hyperboloid model.} For $n\geq2$, The  Minkowski space $\mathbb{L}^{n+2}$, is the vector space $\mathbb{R}^{n+2}$ endowed with the Lorentzian metric $\langle, \rangle$ given by\begin{equation*}
\langle \bar{x},\bar{x}\rangle=-x_0^2+\sum_{i=1}^{n+1}x_i^2,
\end{equation*}
where $\bar{x}=(x_0,x_1,\dots,x_{n+1})\in\mathbb{R}^{n+2}$. Then hyperbolic space, de Sitter spacetime and the positive null cone are given by
\begin{equation*}
\begin{split}
& \mathbb{H}^{n+1}=\{\bar{x}\in\mathbb{L}^{n+2}|\langle \bar{x},\bar{x}\rangle=-1, x_0>0\},\\
&\mathbb{S}^{1,n}=\{\bar{x}\in\mathbb{L}^{n+2}|\langle \bar{x},\bar{x}\rangle=1\},\\
&\mathbb{N}^{n+1}_+=\{\bar{x}\in\mathbb{L}^{n+2}|\langle \bar{x},\bar{x}\rangle=0, x_0>0\},
\end{split}  
\end{equation*}
respectively. We identify the ideal boundary at infinity of hyperbolic space $\mathbb{H}^{n+1}$ with the unit round sphere $\mathbb{S}^n$ sitting at height $x_0=1$ in the null cone $\mathbb{N}^{n+1}_+$ of Minkowski space $\mathbb{L}^{n+2}$.
 Here, horospheres are the intersections of affine {\em null hyperplanes}  of $\mathbb{L}^{n+2}$ with $\mathbb{H}^{n+1}.$
 A {\em null hyperplane} is such that   its normal vector  field   belongs to $\mathbb{N}^{n+1}_+.$
 
\subsection{The $k$-mean curvatures $H_k$}

Let $\Sigma$ be an orientable,  connected, immersed hypersurface in hyperbolic space $\mathbb{H}^{n+1}$. Let 
$\nu$ be an orientation on $\Sigma$ and  denote by $A$ the second fundamental form of the immersion with respect to  $\nu$. Denote by  $\kappa_1,\cdots,\kappa_n$  the principal curvature of $\Sigma$, that is the eigenvalues of $A.$ 
 The $k$-mean curvatures $H_k$ of $\Sigma$,  $1\leq k\leq n$, is defined by 
\begin{equation}
\binom{n}{k}H_k(x)=\sigma_k(\kappa(x)).
\end{equation}

where $\sigma_k: \mathbb{R}^n\to\mathbb{R}$ is  the  $k$-elementary symmetric function  defined by
\begin{equation}
\sigma_k(\lambda_1,\dots,\lambda_n)=\sum_{i_1<i_2<\cdots<i_k}\lambda_{i_1}\cdots \lambda_{i_k}
\end{equation}

Thus, $H_1$ is the mean curvature, $H_n$ is the Gauss-Kronecker curvature and $H_2$ is a multiple of the scalar curvature, when the ambient space is Einstein.
Functions like $\sigma_k$ are a particular case of  hyperbolic polynomials (see   \cite{Garding1959}).

    It  was proved  by R. Reilly in \cite{reilly1973} that  the study of the  $k$-mean curvatures is related to the study of the classical Newton transformations $P_{k}$, that are defined inductively as follows. 
    
$$\begin{array}{l}
\label{Pk}
P_{0}= I,\\
P_{k}=\sigma_{r}I-AP_{k-1},\\
\end{array}$$
where $I$ is the identity matrix and $A$ is a symmetric matrix. Each $P_{k}$ is a self-adjoint operator
 that has the same eigenvectors of $A$.

Before establishing the relation between $P_k$ and $H_k$, let us recall that 
 J. L. Barbosa  and G. Colares extended the relation to space forms  \cite{barbosacolares1997} and M. F. Elbert to any Riemannian manifold \cite{elbert2002} (see also \cite{elbertnelli2019}).

Let $f:\Sigma\longrightarrow {\mathbb H}^{n+1}$ be an isometric immersion of a
connected oriented Riemannian $n$-manifold into the hyperbolic space  and let $A$  its second fundamental form with respect to 
an orientation $\nu.$
Let $D\subset \Sigma$ be a domain. A variation of $D$ is a differentiable
map $F:(-\varepsilon,\varepsilon)\times \bar{D}\longrightarrow
{\mathbb H}^{n+1}$, $\varepsilon >0$, such that for each $t\in 
(-\varepsilon,\varepsilon)$ the map  $F_{t}:\{t\}\times \bar{D}
\longrightarrow{\mathbb H}^{n+1}$ defined by $F_{t}(p)=F(t,p)$ is an immersion
and $F_{0}=f|_{\bar{D}}$. Define $V_{t}(p)=\frac{\partial F}{\partial t}(t,p)$ and  $u(t)=\left<V_{t},\nu_{t}\right >,$ where $\nu_{t}$ is the unit normal vector field in $F_{t}(D)$ such that $\nu_0=\nu.$  We say that a variation $F$ of $D$ has  compact support if ${\rm supp}(F_{t})\subset K$, for all $t\in (-\varepsilon,\varepsilon)$,
where $K\subset D$ is a compact domain.
Let  $H_k^t$ the $k$-mean curvature  of $F_{t}$, and $\sigma_k^t= {\binom{n}{k}} H_k^t.$ Then one has

\begin{equation*}\frac{\partial}{\partial t}(\sigma_{k+1}^t)|_{t=0}=L_{k}(u)+u(\sigma_{1}\sigma_{k+1}-(k+2)\sigma_{k+2}-(n-k)\sigma_k)
+V^T(\sigma_{k+1})
\end{equation*}

where $L_k(u)=tr(P_k({\rm Hess}(u)))$ and $V^T$ is the projection of $V$ on $T\Sigma.$ Notice that, in the case $\sigma_{k+1}$ is constant, then the left-hand side and  the last term in the previous equality are  zero.

 \subsection{Ellipticity  of  $L_k$ and $L_k$-parabolicity}

As $L_k(u)=tr(P_k({\rm Hess}(u)))$,  $L_k$ is  an elliptic operator  if and only if $P_k$  is a positive definite matrix. In particular, $L_0$ is the Laplace-Beltrami operator $\Delta$.
Let us establish  a geometric condition that guarantees the  ellipticity of $L_k.$ 

Denote by $\Gamma_k$ the connected component in $\mathbb{R}^n$ of the set $\{H_k>0\}$ that
contains the vector $(1,\dots,1)$. As it is proved in \cite[Section 2]{fontenele2001tangency}, for any $k=1,\dots,n-1,$

\begin{equation}\label{garding ineq}
  \Gamma_{k+1}\subset \Gamma_k
     \end{equation}

Notice that $\Gamma_n$ is the positive cone in ${\mathbb R}^n.$ Moreover, since $\Gamma_1$ is the largest cone, the mean curvature is positive at any point where the principal curvatures vector stays in the cone $\Gamma_k$.

Moreover, we recall the  classical G$\mathring{\rm a}$rding inequality \cite{Garding1959}:
  \begin{equation}\label{garding-ineq2}
     H_1\geq H_2^{1/2}\geq\cdots\geq H_{k}^{1/k}\geq H_{k+1}^{1/{k+1}}>0,
     \end{equation}
providing all the $r$-mean curvature involved are positive.

\begin{defn}\label{k admissible}
A hypersurface $\Sigma$ of  $\mathbb{H}^{n+1}$ is called $k$-admissible if  the principal curvatures vector at any point of $\Sigma$  stays in the cone $\Gamma_k$, that is,

\begin{equation}
\lambda(x)=(\kappa_1(x),\dots, \kappa_n(x))\in\Gamma_k
\end{equation}
for all $x\in \Sigma$.

\end{defn}

It is well known that  the existence of an elliptic point on a $H_k$-hypersurface with $H_k>0$ yields   that the hypersurface is $k$-admissible \cite{Garding1959, fontenele2001tangency}. Moreover,  $k$-admissibility  yields that  $L_{k-1}$ is elliptic.



 In \cite{alias2013hypersurfaces}, L. Al{\'i}as, D.  Impera and M. Rigoli assumed $L_{k}$-parabolicity  to study the Bernstein type theorems for hypersurfaces with constant $k$-mean curvature in warped product spaces.
 
 \begin{defn}\label{def-parabolic}
 A  hypersurface  $\Sigma$ in $\mathbb{H}^{n+1}$ is $L_{k}$-parabolic if the only bounded above $C^1$ solutions $u:\Sigma\longrightarrow {\mathbb R},$ of the inequality 
 \begin{equation}
 L_k u\geq0
 \end{equation} 
are constants. 
 \end{defn}

    \subsection{Weakly horospherically convexity}

    Intuitively, a hypersurface is {\em weakly horospherically convex} at $p$ if and only if all the principal curvatures of the hypersurface  at $p$ are simultaneously $<-1$ or $>-1$. 
    
            For later use, we recall some basic definitions related to the normal geodesic flow in \cite{bonini2015hypersurfaces, bonini2017weakly}.

Let  $f:M\longrightarrow \mathbb{H}^{n+1}$ be an isometric immersion of an orientable connected Riemannian manifold of dimension $n$, 
and $\eta$ a unit normal vector field orienting $M.$
The {\em hyperbolic Gauss map }  $G$ of $M$ is
 defined as follows: for every $p\in M$, $G(p)$ is the point at infinity of the unique geodesic starting at $f(p)$ with tangent vector $-\eta(p).$

 Notice that $G(p)$ coincides with the point at infinity of the unique horosphere in $\mathbb{H}^{n+1}$ passing through $f(p)$ whose mean curvature vector  coincides with $-\eta(p)$ at $f(p)$.
 Moreover, with our notion of Gauss map,   an  horosphere oriented by the mean curvature vector ($\kappa_i=1$) has injective Gauss map.

   Now, we give a  notion of weak horospherical convexity, using the definition in \cite{bonini2015hypersurfaces, bonini2017weakly} (notice that  the orientation is different from that in \cite{espinar2009hypersurfaces}).
       
    \begin{defn}\cite{bonini2017weakly}
Let $f : M^n\to\mathbb{H}^{n+1}$ be an immersed, oriented hypersurface in $\mathbb{H}^{n+1}$ with unit normal  vector field $\eta.$   
Let $\mathcal{H}_p$ denote the horosphere in $\mathbb{H}^{n+1}$ that is tangent to the hypersurface at $f(p)$ and whose mean curvature vector at $f(p)$ coincides with $-\eta(p)$. We will say that $f : M^n\to\mathbb{H}^{n+1}$ is weakly horospherically convex at $p$ if there exists a neighborhood $V\subset M^n$ of $p$ so that $f(V\setminus\{p\})$ does not intersect with $\mathcal{H}_p$. Moreover, the distance function of the hypersurface $f : M^n\to\mathbb{H}^{n+1}$ to the horosphere $\mathcal{H}_p$ does not vanish up to the second order at $f(p)$ in any direction.
\end{defn}
 
  As we say at the beginning,  the  formal definition of weakly horospherically convex at a point $p$  implies that  all the principal curvatures of the hypersurface  at $p$ are simultaneously $<-1$ or $>-1$. By choosing the orientation, we may assume that all the principal curvatures of  a weakly horospherically convex  hypersurface  are  $>-1$.
We say that  a hypersurface is {\em uniformly}  weakly horospherically convex if  all the principal curvatures $\kappa_i$ are uniformly larger  than $-1$, i.e. $\kappa_i\geq c_0>-1$.

It is clear that  the Gauss map of a weakly horospherically convex  hypersurface is a local diffeomorphism, therefore, such hypersurface can be parametrized by a subset of  $\Omega\subset\mathbb{S}^n$. Now, 
 let $f: \Omega\subset{\mathbb S}^n\to\mathbb{H}^{n+1}$ be a  properly immersed, complete, and uniformly weakly horospherically convex hypersurface.

   The (past) normal geodesic flow $\{f^t\}_{t \in \mathbb{R}}$ in $\mathbb{H}^{n+1}$  of $f$ is  given by
\begin{equation}\label{Eq:NormalFlow}
f^t(x) := \exp_{f(x)}(-t\eta(x)) = f(x) \cosh{t} - \eta(x)\sinh{t}:\Omega \to \mathbb{H}^{n+1} \subset \mathbb{R}^{1,n+1}
\end{equation}
It is well known   \cite{bonini2015hypersurfaces, bonini2017weakly, espinar2009hypersurfaces} that  the principal curvatures $\kappa_i^t$ of $f^t$  are given by
\begin{equation}
\label{kappat}
\kappa_i^t = \frac{\kappa_i + \tanh{t}}{1+\kappa_i \tanh{t}}
\end{equation}
 Moreover, it is easily seen that the hyperbolic Gauss map $G^t$ is invariant under the normal geodesic flow.  
 
  \subsection{ Admissible Weingarten hypersurfaces.}
    
    We will briefly introduce the elliptic problem of Weingarten hypersurfaces \cite{bonini2015hypersurfaces} in our context, that is,  restricted to  weakly  horospherically convex hypersurfaces with the orientation under which  all the principal curvatures are simultaneously larger than $-1$.

    Let $\mathcal{W}(x_1,\cdots,x_n)$  be a symmetric function of $n$-variables, such that  $\mathcal{W}(\kappa_0,\cdots,\kappa_0)=0$ for some number $\kappa_0>-1$. Moreover, let 
 \begin{equation*}
   \mathcal{K}:=\{(x_1,\cdots,x_n)\in\mathbb{R}^n: x_i>-1, i=1,\cdots,n\},
    \end{equation*} 
  
  Let  $\Gamma_+=\{(x_1,\cdots,x_n) : x_1>0,\cdots,x_n>0\}$ and   $\Gamma^*$ be  an open connected component of  
 $\{(x_1,\cdots,x_n) : \mathcal{W}(x_1,\cdots,x_n)>0\}$
satisfying
   \begin{enumerate}[leftmargin=*]
   \item \label{positive}
$(\kappa,\cdots,\kappa)\in\Gamma^*\cap\mathcal{K},$ for every $\kappa\in (\kappa_0,\infty), $
    \item \label{connected}
    For every  $(x_1,\cdots,x_n)\in\Gamma^*\cap\mathcal{K},$  and  $(y_1,\cdots,y_n)\in \Gamma^*\cap\mathcal{K}\cap((x_1,\cdots,x_n)+\Gamma_+),$ 
   there exists a curve
$ \gamma $ connecting  $(x_1,\cdots,x_n)$  to $(y_1,\cdots,y_n)$ inside $\Gamma^*\cap\mathcal{K}$
 such that $\gamma'\in\Gamma_+ $ along  $\gamma$
  \item\label{elliptic}
$\mathcal{W}\in C^1(\Gamma^*)$ and $\frac{\partial \mathcal{W}}{\partial x_i}>0$ in $ \Gamma^*.$
\end{enumerate} 
Suppose $\Sigma$ is a hypersurface of $\mathbb{H}^{n+1}$  satisfying  the following  {\em general Weingarten equation}
\begin{equation}\label{weingarten}
   \mathcal{W}(\kappa_1,\cdots,\kappa_n)=K \ \text{and} \ (\kappa_1,\cdots,\kappa_n)\in \Gamma^*\cap\mathcal{K} \ \text{on} \ \Sigma,
   \end{equation} 
for some positive constant $K$, where $(\kappa_1,\cdots,\kappa_n)$ are the principal curvatures of the $\Sigma$.

 \begin{defn}
   In \eqref{weingarten}, a positive number $K$ is admissible for a given curvature function $\mathcal{W}$ if $\mathcal{W}(\bar\kappa_0,\cdots,\bar\kappa_0)=K$, $\frac{\partial\mathcal{W}}{\partial x_i}(\bar\kappa_0,\cdots,\bar\kappa_0)>0$, and $\bar\kappa_0>\kappa_0$.
   \end{defn}

  A  hypersurface  $\Sigma$ such that   \eqref{weingarten} is satisfied for an admissible constant is called {\em admissible Weingarten hypersurface}. If the principal curvatures have a uniform lower bound which is strictly bigger than $-1$, then we call it uniformly admissible Weingarten hypersurface.
In particular if $ \mathcal{W}$ is an elementary  symmetric function of the principal curvatures, all the assumptions are satisfied.
Hence a weakly horospherical convex, $r$-admissible $H_r$-hypersurface is an admissible Weingarten hypersurface.
   We will always   chose $K$ such that 
   \begin{equation}
   \label{K-adm}
   \mathcal{W}(\bar\kappa_0,\cdots,\bar\kappa_0)=K, \ \frac{\partial\mathcal{W}}{\partial x_i}(\bar\kappa_0,\cdots,\bar\kappa_0)>0
   \end{equation}
    for some  $\bar\kappa_0>\kappa_0$.

\section{Bernstein theorem for embedded hypersurfaces}

In this section, we extend the Bernstein type  theorem by Do Carmo and Lawson in \cite{do1983alexandrov} to 
hypersurfaces with constant $r$-mean curvature in the hyperbolic space. 
 Recall that $\partial_{\infty} {\mathbb H}^{n+1}$ has a natural conformal structure of a sphere ${\mathbb S}^n(\infty).$ When the
 asymptotic boundary of a hypersurface $\Sigma$ is a sphere in ${\mathbb S}^n(\infty),$ we can assume that it is an equator. We say that 
 $\Sigma$ separates poles if  the north and the south poles with respect to  such equator are in distinct connected components of 
 ${\mathbb H}^{n+1}\cup{\mathbb S}^n(\infty)\setminus(\Sigma\cup\partial_{\infty}\Sigma).$
 
\begin{thm}\label{DL r-mean}
Let  $\Sigma$ be a complete hypersurface properly embedded in hyperbolic space $\mathbb{H}^{n+1}$ with constant $r$-mean curvature $(r\geq2)$. Denote by $\partial_\infty \Sigma\subset {\mathbb S}^n(\infty)$ the asymptotic boundary of $\Sigma$. Then we have the following:
\begin{enumerate}[leftmargin=*]
\item  if $\partial_\infty \Sigma$ is a point, then $\Sigma$ is a horosphere;

\item  if $\partial_\infty \Sigma$ is a sphere and $\Sigma$ separates poles, then $\Sigma$ is an equidistant sphere.
\end{enumerate}
\end{thm}

\begin{proof} {\em (1)} Suppose the asymptotic boundary of $\Sigma$ is only one point $q_\infty\in{\mathbb S}^n(\infty)$. First, inspired by \cite{nelli1995some}, we prove that $\Sigma$ has a strictly convex point. We consider the half-space model for $\mathbb{H}^{n+1}$ so that $q_\infty$ corresponds to the infinity point. In this model, the horospheres whose asymptotic boundary is  $q_\infty$ are given by the equations $x_{n+1}=\text{constant}$. We write the coordinates as $(\bar{x}, x_{n+1})$, where $\bar{x}=\{x_1,\cdots,x_n\}$. Then the geodesics orthogonal to the horospheres with $q_{\infty}$ as asymptotic point, are one-to-one correspondence with the points $\bar{x}\in\mathbb{R}^{n}$ and can be written as $\gamma_{\bar{x}}(s)=(\bar{x},s)$ for $s>0$. Each such geodesic $\gamma=\gamma_{\bar{x}}$ determines a family of hyperplanes orthogonal to $\gamma,$  $h_t(\bar{x})=\{x\in\mathbb{R}^{n+1}_+:\|x-(\bar{x},0)\|=t\}$, where $\|\cdot\|$ denotes the standard Euclidean norm.
Let ${\mathcal E}_t$ be a family of  equidistant spheres  such that for any $t$, the asymptotic boundary is $\partial_{\infty}{\mathcal E}_t=\{(\bar x,0)\in{\mathbb R}^{n+1}\ |\ x_1^2+\dots+x_n^2=t^2\}$ and that the mean curvature vector at the highest point, points upward.
Notice that every principal curvature of ${\mathcal E}_t$ is equal to  a constant $0<H_0<1$ at any point. 
For $t$ small, ${\mathcal E}_t\cap \Sigma=\emptyset.$ Then  increase $t$ till the first $\bar t$ such that 
${\mathcal E}_{\bar t}\cap \Sigma$ contains a point $p.$ Notice that, in a neighborhood of $p$, the hypersurface $\Sigma$ lies above 
${\mathcal E}_{\bar t}.$ For any tangent vector $X$ at $p$ (tangent to  $\Sigma$ and ${\mathcal E}_{\bar t}$), consider the 2-plane $P_X$ generated by $X$ and the $x_{n+1}$-axis. In a neighborhood of $p$, $P_X\cap\Sigma$ is a regular curve, that lies above the regular curve$P_X\cap{\mathcal E}_{\bar t}$ that has curvature $H_0.$ Hence the curvature of $P_X\cap\Sigma$ is larger or equal to $H_0$. 
Then, $p$ is a strictly convex point of $\Sigma.$

Since $\partial_\infty\Sigma=\{q_{\infty}\}$,   if $t$ small enough, we have 

\begin{equation*}
h_t\cap\Sigma=\emptyset
\end{equation*}
 For any $t>0$, we denote by $\mathcal{H}^{n+1}_+(t)$ and $\mathcal{H}^{n+1}_-(t)$ the half-spaces determined by $h_t=h_t(\bar x)$. We set 
\begin{equation*}
\Sigma_{\pm}(t)=\Sigma\cap\mathcal{H}^{n+1}(t)
\end{equation*}
and $\Sigma_-(t)=\emptyset$ for  any $\bar x$  and $t$ sufficiently small.

Note also that $\Sigma$ separates $\mathbb{H}^{n+1}$ into two connected components $\Omega_+$ and $\Omega_-$ where $\partial_{\infty}\Omega_+=\{q_{\infty}\}$ and $\partial_\infty\Omega_-\cong\{\mathbb{R}^n\times\{0\}\}$.

Let $t_0$ be the smallest $t$ for which $h_t(\bar x)\cap\Sigma\neq\emptyset$. Then for all $t>t_0$ such that $t-t_0$ sufficiently small, consider the reflected hypersurfaces $\Sigma^{'}_{-}(t)=r_{h_t(\bar x)}(\Sigma_-(t))$, where $r_{h_t}$ is the hyperbolic isometry fixing $h_t(\bar x).$

 The hypersurfaces $\Sigma^{'}_{-}(t)$  have the following properties:
\begin{equation}\label{reflect}
\Sigma_{-}^{'}(t)\subset\Omega_+,
\end{equation}
\begin{equation}
\Sigma_{-}^{'}(t)-\partial\Sigma_{-}^{'}(t)\subset\text{int}(\Omega_+).
\end{equation}

Let us now suppose that condition \eqref{reflect} does not hold for some large $t$. Then there must be some $\bar{t}>t_0$ such that the surfaces $\Sigma_{-}^{'}(\bar{t})$ and $\Sigma_{+}(\bar{t})$
have a point $p$ of common tangency (possibly at the boundary), and that $\Sigma_{-}^{'}(\bar{t})$
lies above $\Sigma_{+}(\bar{t})$ in a neighborhood of $p$. By  the tangency  principle in \cite[Theorem 1.1]{fontenele2001tangency}, we conclude that these hypersurfaces coincide. From this, it easily follows that $\Sigma$ is compact, which is a contradiction. Thus we conclude that \eqref{reflect} holds for all $t$ and all $\bar{x}\in\mathbb{R}^n$.

Relation \eqref{reflect}  yields  that at any of its points, $\Sigma$ is tangent to the (horizontal) horosphere passing through the point. 
If this was not the case, then it is easily seen that there is a geodesic hyperplane $h_t(\bar x)$ (for an appropriate choice of $\bar{x}$ and $t$) such that $\Sigma_{-}^{'}(t)\not\subset\Omega_+$ and this contradicts \eqref{reflect}. It follows that the function $x_{n+1}$ must be constant on $\Sigma$ and $\Sigma$ is a horosphere.

{\em (2)} First we prove that $\Sigma$ has a strictly convex point. By the embeddedness, $\Sigma$ divides the hyperbolic space into two connected components. Also, $\partial_\infty\Sigma$ separates $S^n(\infty)$ into two components: $S^n_+\cup S^n_-=S^n(\infty)-\partial_\infty\Sigma$. Denote by $\nu$ the unit normal orienting $\Sigma$.
We first prove that $\Sigma$ has a strictly convex point, unless $\Sigma$ is a hyperplane.

Let $N_{\varepsilon}(t)$ $t\in[0,1],$  be a family of equidistant spheres with the following properties.
\begin{enumerate}[leftmargin=*]
\item The mean curvature vector of $N_{\varepsilon}(t)$ points upward for any $t\in[0,1].$
\item The angle between $N_{\varepsilon}(0)$ and $\{x_{n+1}=0\},$ is $\frac{\pi}{2}+\varepsilon$ 
with  $\varepsilon>0$ small and  $\Sigma$ is contained in the mean-convex side of $N_{\varepsilon}(0).$ 
\item $N_{\varepsilon}(t)$ is obtained from $N_{\varepsilon}(0)$ by a homothety from the euclidean center of $N_{\varepsilon}(0).$ 
By construction, the angle $\theta_{\varepsilon}(t)$ between $N_{\varepsilon}(t)$ and $\{x_{n+1}=0\}$  satisfies  
$\theta_{\varepsilon}(t)>\frac{\pi}{2}$ hence  the mean curvature vector of $N_{\varepsilon}(t)$ points towards $\Sigma.$
\item $\partial_{\infty} N_{\varepsilon}(1)=\partial_{\infty}\Sigma.$
\end{enumerate}

Increasing $t$, there exists a  first $\bar t<1$  such that $N_{\varepsilon}(\bar t)$  and $\Sigma$ has a  contact point $p$, then $p$ is a strictly convex point. If such point does not exists, then $\Sigma$ lies above $N_{\varepsilon}(1).$ Notice that, letting 
$\varepsilon\longrightarrow 0,$ $N_{\varepsilon}(1)$ tends to the hyperplane whose asymptotic boundary coincides with $\partial\Sigma.$
Hence $\Sigma$ lies above such hyperplane.

Let $S_{\varepsilon}(t)$ $t\in[0,1],$  be a family of equidistant spheres with the following properties.
\begin{enumerate}[leftmargin=*]
\item The mean curvature vector of $S_{\varepsilon}(t)$, at the highest point, points downward for any $t\in[0,1].$
\item The angle between $\partial_{\infty}S_{\varepsilon}(0)$ and $\{x_{n+1}=0\},$ is $\frac{\pi}{2}-\varepsilon$ 
with  $\varepsilon>0$ small and  $\Sigma$ is contained in the mean-convex side of $S_{\varepsilon}(0).$ 
\item $S_{\varepsilon}(t)$ is obtained from $S_{\varepsilon}(0)$ by a homothety from the euclidean center of $S_{\varepsilon}(0).$ 
By construction, the angle $\theta_{\varepsilon}(t)$ between $S_{\varepsilon}(t)$ and $\{x_{n+1}=0\}$  satisfies  
$\theta_{\varepsilon}(t)<\frac{\pi}{2}$ hence  the mean curvature vector of $S_{\varepsilon}(t)$ points towards $\Sigma$.
\item $\partial_{\infty} S_{\varepsilon}(1)=\partial\Sigma.$
\end{enumerate}

Increasing $t$, there exists a  first $\bar t<1$  such that $S_{\varepsilon}(\bar t)$  and $\Sigma$ has a  contact point $p$, then $p$ is a strictly convex point. If such point does not exists, then $\Sigma$ lies below $S_{\varepsilon}(1).$ Notice that, letting 
$\varepsilon\longrightarrow 0,$ $S_{\varepsilon}(1)$ tends to the hyperplane whose asymptotic boundary coincides with $\partial\Sigma.$
Hence $\Sigma$ lies below such hyperplane.

We conclude that either there is a strictly convex point on $\Sigma$ or $\Sigma$ is a hyperplane.

Hence we may  assume that there is a strictly convex point. By the argument after  \eqref{garding ineq}, $H_1$ is positive on $\Sigma$.
Moreover, since  $\Sigma$ separates poles,  we can  select $S^n_+$  as asymptotic boundary of  the region into which the mean curvature vector of $\Sigma$ points. 

Now let us use the half-space model.  We take the center of $S^-_n$ to be the origin of the half-space model and then $S^n_+$ is the component which is unbounded in the Euclidean topology.

 First we prove that  the $r$-mean curvature of $\Sigma$ satisfies $H_r< 1.$
Fix a point $x\in S^n_+$ and consider the family of horospheres having $x$ as asymptotic point. There is a horopshere first touches 
$\Sigma.$ At this contact point, the horosphere and $\Sigma$  are  tangent. Moreover, with respect to that normal vector, $\Sigma$ is below the horosphere. Therefore $r$-mean curvature $H_r$ of $\Sigma$ is strictly less than $1$ by the tangency principle in \cite{fontenele2001tangency}. 

Let ${\mathcal E}$ be  the equidistant sphere of $r$-mean curvature $H_r,$ such that 
$\partial_{\infty}{\mathcal E}=\partial_\infty\Sigma$ and the mean curvature vector of ${\mathcal E}$ also points to $S^n_+$. By applying to  ${\mathcal E}$ the isometries of  ${\mathbb H}^{n+1}$ given by homotheties  with respect to the center of $S^n_-$, we get a foliation of $\mathbb{H}^{n+1}$ consisting of equidistant spheres,  denoted by ${\mathcal E}_t,$ $t\in {\mathbb R}.$ Choose the parameter $t$ such that 
${\mathcal E}_0={\mathcal E}$ and  ${\mathcal E}_t$ goes to the origin as $t\to-\infty$ (to the infinity point as $t\to+\infty$). Since $\partial_{\infty}{\mathcal E}=\partial_\infty\Sigma$, we have that $\Sigma\cap {\mathcal E}_t$ is compact for all $t\neq0$ and $\Sigma\cap {\mathcal E}_t=\emptyset$ for all $|t|$ sufficiently large. If $\Sigma\neq {\mathcal E}$, then $\Sigma\cap {\mathcal E}_{\bar t}\neq\emptyset$ for some $\bar{t}\neq0$. Suppose $\bar{t}>0$, let $t_1=\sup\{t: \Sigma\cap {\mathcal E}_{\bar t}\neq\emptyset\}$. Then $\Sigma$ is below ${\mathcal E}_{t_1}$ with respect to $\nu$ near the contact point. Thus, by the tangency principle in \cite{fontenele2001tangency}, $\Sigma={\mathcal E}_{t_1}$ which contradicts $\partial_{\infty}{\mathcal E}=\partial_\infty\Sigma$. The case of $\bar{t}<0$ is similar. Hence, $\Sigma={\mathcal E}$. 

\end{proof}

\section{Bernstein theorem for immersed hypersurface}

In this section, we consider the Bernstein theorem for immersed hypersurfaces, either with constant  $r$-mean curvature, or satisfying a general elliptic Weingarten  equation. 

\subsection{The case  of constant $r$-mean curvature hypersurface contained in a slab}

In this section, we consider the Bernstein theorem for immersed hypersurfaces with constant $r$-mean curvature contained in a slab. 
The starting point is the following non-existence theorem.

\begin{thm}\label{nonexistence}
In any slab of hyperbolic space $\mathbb{H}^{n+1}$, there is no complete properly immersed hypersurface $\Sigma$ with $r$-mean curvature satisfying $\Bar{H}:=\sup_{\Sigma}|H_r|<1$ for any $r\geq1$.
\end{thm}

\begin{proof}
In the upper half-space model of hyperbolic space $\mathbb{H}^{n+1}$,  we assume that the slab is between two horospheres given by two horizontal Euclidean hyperplanes. We can foliate the whole space by a family of  equidistant spheres ${\mathcal E}(t)$ with mean curvature being $\bar{H}^{\frac1r}$ with respect to the  normal vector field that points upward at the highest point,   for $t\in[0,\infty)$. When $t$ is small, ${\mathcal E}(t)$ and $\Sigma$ are disjoint. Then,  consider  $t_0$ such that  ${\mathcal E}(t_0)$ and  the hypersurface $\Sigma$ first touch  at some point $p$. In a neighborhood of $p$, $\Sigma$ is above ${\mathcal E}(t_0)$ with respect  the upward normal vector of ${\mathcal E}(t_0)$ at $p$. We also have $H_r({\mathcal E}(t_0))\geq H_r(\Sigma)$ in that neighborhood and ${\mathcal E}(t_0)$ is $r$-admissible. Therefore, by the tangency principle \cite[Theorem 1.1]{fontenele2001tangency}, $\Sigma={\mathcal E}(t_0)$, which is a contradiction, because ${\mathcal E}(t_0)$ is not contained in any slab.
\end{proof}

As a consequence of the previous Theorem, we are able to  prove the analogous of   \cite[Theorem 1]{alias2006uniqueness} 
for surfaces of  constant Gaussian curvature. 

\begin{cor}
If $\Sigma$ is a properly immersed complete surface  in ${\mathbb H}^3$ with constant $2$-mean curvature $0<H_2\leq1$ contained in a slab then $\Sigma$ is a horosphere.
\end{cor}

\begin{proof}
 By Theorem \ref{nonexistence}, we have $H_2=1$, which implies that $\Sigma$ is a complete flat immersion. Then the result follows from  \cite[Theorem 5]{galvez2009surfaces}.
\end{proof}

 In higher dimension, we need to add  $L_{k}$-parabolicity (see Definition \ref{def-parabolic}) and a geometric assumption,  in order to get a Bernstein type Theorem.

\begin{thm}\label{theorem-parabolic}
Let $\Sigma$ be a complete, $r$-admissible, $L_{r-1}$-parabolic properly immersed hypersurface with constant $r$-mean curvature. If  $\Sigma$  is contained in a slab  and  the angle function does not change sign, then $\Sigma$ is a horosphere.
\end{thm}

\begin{proof} 

  Define $\phi=e^hH_r^{1/r}+e^h\Theta$ where $\Theta$ is the angle function and $h$ is the height, as we defined in Section 2.1. It follows from the proof of Theorem 32 in \cite{alias2013hypersurfaces} that 
\begin{equation}
\begin{split}
L_{r-1}\phi\geq&c_{k-1}e^hH_r^{1/r}(H_{r-1}-H_r^{\frac{r-1}{r}})\\
&-\binom{n}{k}e^h\Theta(nH_1H_r-(n-r)H_{r+1}-rH_r^{\frac{r+1}{r}})\geq0,
\end{split}
\end{equation}
where we have used the Garding inequality \eqref{garding-ineq2} in the last inequality. Since $\Sigma$ is $L_{r-1}$-parabolic, $\phi$ is a constant. In particular, $\Delta\phi=0.$ Then,   equation (3.8) in \cite{alias2007constant} gives that
\begin{equation*}
0=\Delta \phi=e^h\Theta(||A||^2-(n-1)H^2)
\end{equation*}

Notice that, in the notation of  (3.8) in \cite{alias2007constant},  $\rho(t)=e^t,$ ${\mathcal H}(t)=1$ and $Ric_{\mathbb P} (\hat N)=0$ as ${\mathbb P}={\mathbb R}^n.$

 We conclude that $\|A\|^2=(n-1)H^2,$ that yields $\Sigma$ is a totally umbilical hypersurface.  Thus, all the principal curvatures are equal to a constant. Moreover, since the hypersurface is contained in a slab, then all the principal curvatures are  larger  or equal  than $1$ by Theorem \ref{nonexistence}. Then it is either a horopshere or sphere by  \cite[Theorems A,B]{currier1989hypersurfaces}. However, in the latter case, the angle function changes sign. Therefore, it has to be a horosphere.

\end{proof}

\subsection{The case of admissible Weingarten hypersurfaces} 
In this section, we get a Bernstein type theorem for admissible Weingarten  hypersurfaces.
 Note  that uniformly weakly horospherically convex hypersurfaces  with injective hyperbolic Gauss map become embedded under the (past) normal geodesic flow (see Theorem 1.3 in \cite{bonini2017weakly}).

\begin{thm}\label{weakly horospherically convex}
Suppose that $ \Sigma$ is an immersed, complete, uniformly  weakly horospherically convex  admissible Weingarten hypersurface in $\mathbb{H}^{n+1}$.  Then $\Sigma$  is a horosphere provided its asymptotic boundary is a single point.
\end{thm}

\begin{proof}
As the Gauss map of $\Sigma$  is locally injective, we can apply   Theorem 4.2 in \cite{bonini2017weakly}. Then,  for $t$ large enough,   the past normal geodesic flow defined in \eqref{Eq:NormalFlow},  deforms $\Sigma$ into a properly embedded, uniformly weakly horospherically convex hypersurface $\Sigma_t$ with single point boundary at infinity. Moreover,   the principal curvatures of $\Sigma_t$ are given by (see  \eqref{kappat}):

\begin{equation}
\kappa_i^t=\frac{\kappa_i + \tanh(t)}{1+\kappa_i \tanh(t)} \ \text{and} \ \kappa_i=\frac{\kappa_i^t-\tanh(t)}{1-\kappa_i^t\tanh(t)}
\end{equation}
 Let 
 \begin{equation}
 \mathcal{W}^t(x_1,\cdots,x_n):=\mathcal{W}\bigg(\frac{x_1-\tanh(t)}{1-x_1\tanh(t)},\cdots,\frac{x_n-\tanh(t)}{1-x_n\tanh(t)}\bigg).
 \end{equation}
 Then  it follows from the definition of $\mathcal{W}$ that $\mathcal{W}^t$ is a symmetric function of $n$-variables with 
 \begin{equation}
 \mathcal{W}^t\left(\frac{\kappa_0 + \tanh(t)}{1+\kappa_0 \tanh(t)},\dots,\frac{\kappa_0 + \tanh(t)}{1+\kappa_0 \tanh(t)}\right)= \mathcal{W}(\kappa_0,\dots,\kappa_0)=0
 \end{equation}
 and $\frac{\kappa_0 + \tanh(t)}{1+\kappa_0 \tanh(t)}>-1$.
 
 Define
 \begin{equation}
 \mathcal{T}(x_1,\cdots,x_n)=\bigg(\frac{x_1+\tanh(t)}{1+x_1\tanh(t)},\cdots,\frac{x_n+\tanh(t)}{1+x_n\tanh(t)}\bigg).
\end{equation}
We then have
\begin{equation}
\Gamma_t^*\cap\mathcal{K}=\mathcal{T}(\Gamma^*\cap\mathcal{K})
\end{equation}
and 
\begin{equation}
\mathcal{T}((x_1,\cdots,x_n)+\Gamma_n)=\mathcal{T}(x_1,\cdots,x_n)+\Gamma_n.
\end{equation}
For ellipticity, one can easily compute
\begin{equation}
\frac{\partial \mathcal{W}}{\partial x_i}=\frac{1-\tanh^2(t)}{(1-x_i\tanh(t))^2}\frac{\partial\mathcal{W}}{\partial y_i}.
\end{equation}
Therefore, $(\mathcal{W}_t,\Gamma_t^*)$ satisfies \eqref{positive}-\eqref{elliptic}. Thus,  \cite[Theorem 4.4]{bonini2015hypersurfaces}
yields that $\Sigma_t$ is a horosphere  and so is $\Sigma$, since $\Sigma$ is a time-slice of the foliation formed by a horosphere under the normal geodesic flow.

\end{proof}

\begin{cor}\label{cor-weingarten}
Suppose that $ \Sigma$ is an immersed, complete, uniformly weakly horospherically convex, $r$-admissible $H_r$-hypersurface in $\mathbb{H}^{n+1}$.  Then $\Sigma$  is a horosphere provided its asymptotic boundary is a single point.
\end{cor}

\section{$r$-mean curvature rigidity of horospheres and equidistan spheres}

Motivated by the following  result of M. Gromov \cite{gromov2018}:

\

{\em A hyperplane in a Euclidean space $\mathbb{R}^n$ cannot be perturbed on a compact set so
that its mean curvature satisfies $H\geq0$.}

\

R. Souam proved the  following  extension to hyperbolic space  \cite{souam2019mean}:

\

{\em Let $M$ denote a horosphere, an equidistant sphere or a hyperplane in a hyperbolic
space $\mathbb{H}^{n+1}$, $n\geq2$ and $H_M\geq 0$ its constant mean curvature. Let $\Sigma$ be a connected properly embedded $C^2$-hypersurface in $\mathbb{H}^{n+1}$ which coincides with $M$ outside a compact subset of $\mathbb{H}^{n+1}$. If the mean curvature of $\Sigma$ is $\geq H_M$, then  $\Sigma=M$. }

\

The proof in \cite{souam2019mean} is based on the tangency principle for mean curvature. We are able to extend Souam's result to the case of $r$-mean curvature by the tangency principle in \cite{fontenele2001tangency}.

\begin{thm}\label{rigidity}
Let $M$ be a horosphere or an equidistant sphere in hyperbolic space $\mathbb{H}^{n+1}$, $n\geq2$  and denote by  $H_M>0$  its $r$-mean curvature, $r\geq1$, with respect to the orientation given by the mean curvature vector. Let $\Sigma$ be a connected properly embedded  $C^2$ hypersurface in $\mathbb{H}^{n+1}$ which coincides with $M$ outside a compact subset $B$ in $\mathbb{H}^{n+1}$. Choose the orientation on $\Sigma$ such that the $r$-mean curvature $H_r$ of $\Sigma$ is equal to $H_M$ outside the compact set $B.$  With respect to this orientation, if either $H_r\geq H_M$ or $|H_r|\leq H_M$, then $\Sigma\equiv M$.
\end{thm}

\begin{proof}
We take the upper half-space model of hyperbolic space.

(1)  The case of horospheres.
Consider the family of  horospheres  ${\mathcal O}_t=\{x\in\mathbb{R}^{n+1}|x^{n+1}=t\},$ $t>0$ and assume  that $M={\mathcal O}_1$. Notice that, $H_M=1.$

Let $\Sigma$ be a connected properly embedded  $C^2$ hypersurface in $\mathbb{H}^{n+1}$ with $r$-mean curvature either $H_r\geq 1$ or $H_r\leq 1$. 

 Now, assume that $H_r\geq 1$.

Since $\Sigma$ coincides with $M$ outside a compact subset $B$ of $\mathbb{H}^{n+1}$, the mean convex side of $\Sigma,$ that is  the component where the mean curvature vector points  towards,  coincides with the domain $\{x\in\mathbb{R}^{n+1}|x^{n+1}>1\},$ outside  a compact set.

We consider the largest $T\geq 1$ such that $\Sigma\cap{\mathcal O}_T\not=\emptyset$ and let $p\in \Sigma\cap{\mathcal O}_T$.
At the point $p$, $\Sigma$ and ${\mathcal O}_T$ are tangent, in a neighborhood of $p$, the horosphere ${\mathcal O}_T$ lies above $\Sigma,$ while the $r$-mean curvature of $\Sigma$   is larger or equal than the $r$-mean curvature of ${\mathcal O}_T$ (with respect to the upward normal vector). By the tangency  principle  \cite[Theorem 1.1]{fontenele2001tangency}, $\Sigma$ coincide with ${\mathcal O}_T$ in a open neighborhood of $p$.
Hence the subset $\Sigma\cap{\mathcal O}_T$ is open. As it is also closed, we get that $\Sigma$ coincides with ${\mathcal O}_T$ and $T=1$.

Now, assume that  $H_r\leq 1$. 
We consider the smallest $\tau\leq 1$ such that $\Sigma\cap{\mathcal O}_{\tau}\not=\emptyset$ and let $p\in \Sigma\cap{\mathcal O}_{\tau}$. 
At the point $p$, $\Sigma$ and ${\mathcal O}_{\tau}$ are tangent, in a neighborhood of $p$, the horosphere ${\mathcal O}_{\tau}$ lies below $\Sigma$, while the $r$-mean curvature of $\Sigma$  is less than or equal to the $r$-mean curvature of ${\mathcal O}_T$ (with respect to the upward normal vector). Notice that, we do not know in advance whether the normal vector to $\Sigma$ at $p$ points upward or downward. However,  by the assumption $|H_r|\leq H_M$, one can check that  the $r$-mean curvature of $\Sigma$ with respect to the upward normal is always less than or equal to $H_M$.  Then it follows from the tangency  principle  \cite[Theorem 1.1]{fontenele2001tangency} that  $\Sigma$ coincide with ${\mathcal O}_{\tau}$ in a open neighborhood of $p$. Hence the subset $\Sigma\cap{\mathcal O}_{\tau}$ is open. As it is also closed, we get that $\Sigma$ coincides with ${\mathcal O}_{\tau}$ and $\tau=1$. Notice that, in spite of the fact that $\Sigma$ may have non positive $r$-mean curvature, we can apply   \cite[Theorem 1.1]{fontenele2001tangency} because the principal curvature vector of ${\mathcal O}_{\tau}$ lies in the positive cone.

(2)  The case of equidistant spheres.

 We may assume that the mean curvature vector of $M$ points upward and that the  asymptotic boundary of $M$ is a $(n-1)$-sphere of  $\partial_{\infty}{\mathbb H}^{n+1}$ centered at the origin of ${\mathbb R}^{n+1}.$    As in the previous case, since $\Sigma$ coincides with $M$ outside a compact subset $B$ of $\mathbb{H}^{n+1}$, the mean convex side of $\Sigma,$  coincides with the mean convex side of $M$ outside  a compact set.
 Let   ${\mathcal E}(t)$  be a foliation in equidistant spheres, obtained rescaling $M$, with respect to the origin, such that  ${\mathcal E}(0)=M$ and ${\mathcal E}(t)$ is above $M$ for positive $t$ and below $M$ for negative  $t.$  Notice that, all the ${\mathcal E}(t)$  has the same $r$-mean curvature.
 Similar as in (1), when $H_r\geq H_M$, $\Sigma$ touches ${\mathcal E}(t)$, $t\geq 0$  from below. When $|H_r|\leq H_M$, $\Sigma$ touches ${\mathcal E}(t)$, $t\leq 0$ from above. In both cases, by tangency principle, $\Sigma=M$.

\end{proof}


\nocite{*}


\bibliographystyle{amsplain}
\bibliography{reference}

\providecommand{\bysame}{\leavevmode\hbox to3em{\hrulefill}\thinspace}
\providecommand{\MR}{\relax\ifhmode\unskip\space\fi MR }
\providecommand{\MRhref}[2]{%
  \href{http://www.ams.org/mathscinet-getitem?mr=#1}{#2}
}
\providecommand{\href}[2]{#2}
\begin{thebibliography}{10}

\bibitem{alias2006uniqueness}
Luis~J. Al{\'\i}as and Marcos Dajczer, \emph{Uniqueness of constant mean
  curvature surfaces properly immersed in a slab}, Commentarii Mathematici
  Helvetici \textbf{81} (2006), no.~3, 653--663.

\bibitem{alias2007constant}
\bysame, \emph{Constant mean curvature hypersurfaces in warped product spaces},
  Proceedings of the Edinburgh Mathematical Society \textbf{50} (2007), no.~3,
  511--526.

\bibitem{alias2013hypersurfaces}
Luis~J. Al{\'\i}as, Debora Impera, and Marco Rigoli, \emph{Hypersurfaces of
  constant higher order mean curvature in warped products}, Transactions of the
  American Mathematical Society \textbf{365} (2013), no.~2, 591--621.

\bibitem{almgren1966some}
Frederick~J. Almgren, \emph{Some interior regularity theorems for minimal
  surfaces and an extension of {B}ernstein's theorem}, Annals of Mathematics
  (1966), 277--292.

\bibitem{barbosacolares1997}
Jo{\~a}o Lucas~Marques Barbosa and Ant{\^o}nio~Gervasio Colares,
  \emph{Stability of hypersurfaces with constant $r$-mean curvature}, Annals of
  Global Analysis and Geometry \textbf{15} (1997), no.~3, 277--297.

\bibitem{BDG}
Enrico Bombieri, Ennio De~Giorgi, and Enrico Giusti, \emph{Minimal cones and
  the {B}ernstein problem}, Inventiones Mathematicae \textbf{7} (1969),
  243--268.

\bibitem{bonini2015hypersurfaces}
Vincent Bonini, Jos{\'e}~M. Espinar, and Jie Qing, \emph{Hypersurfaces in
  hyperbolic space with support function}, Advances in Mathematics \textbf{280}
  (2015), 506--548.

\bibitem{bonini2017weakly}
Vincent Bonini, Jie Qing, and Jingyong Zhu, \emph{Weakly horospherically convex
  hypersurfaces in hyperbolic space}, Annals of Global Analysis and Geometry
  \textbf{52} (2017), no.~2, 201--212.

\bibitem{bryant1987}
Robert Bryant, \emph{Surfaces of mean curvature one in hyperbolic space},
  Ast{\'e}risque \textbf{154} (1987), no.~155, 321--347.

\bibitem{currier1989hypersurfaces}
Robert~J. Currier, \emph{On hypersurfaces of hyperbolic space infinitesimally
  supported by horospheres}, Transactions of the American Mathematical Society
  \textbf{313} (1989), no.~1, 419--431.

\bibitem{de1965estensione}
Ennio De~Giorgi, \emph{An extension of {B}ernstein's theorem}, Annals of the
  Scuola Normale Superiore of Pisa-Class of Sciences \textbf{19} (1965), no.~1,
  79--85.

\bibitem{do1983alexandrov}
Manfredo~P. Do~Carmo and H.~Blaine Lawson, \emph{On {A}lexandrov-{B}ernstein
  theorems in hyperbolic space}, Duke Math. J \textbf{50} (1983), no.~4,
  995--1003.

\bibitem{elbert2002}
Maria~Fernanda Elbert, \emph{Constant positive 2-mean curvature hypersurfaces},
  Illinois J. of Math. \textbf{1} (2002), no.~46, 247--267.

\bibitem{elbertnelli2019}
Maria~Fernanda Elbert and Barbara Nelli, \emph{A note on the stability for
  constant higher mean curvature hypersurfaces in a {R}iemannian manifold},
  arXiv preprint arXiv:1912.12103 (2019).

\bibitem{espinar2009hypersurfaces}
Jos{\'e}~M. Espinar, Jos{\'e}~A. G{\'a}lvez, and Pablo Mira,
  \emph{Hypersurfaces in $\mathbb{R}^{n+1}$ and conformally invariant
  equations: the generalized {C}hristoffel and {N}irenberg problems}, Journal
  of the European Mathematical Society \textbf{11} (2009), no.~4, 903--939.

\bibitem{fleming1962oriented}
Wendell~H. Fleming, \emph{On the oriented {P}lateau problem}, Rendiconti del
  Circolo Matematico di Palermo \textbf{11} (1962), no.~1, 69--90.

\bibitem{fontenele2001tangency}
Francisco Fontenele and S{\'e}rgio~L. Silva, \emph{A tangency principle and
  applications}, Illinois Journal of Mathematics \textbf{45} (2001), no.~1,
  213--228.

\bibitem{galvez2009surfaces}
Jos{\'e}~A. G{\'a}lvez, \emph{Surfaces of constant curvature in 3-dimensional
  space forms}, Mat. Contemp \textbf{37} (2009), 1--42.

\bibitem{Garding1959}
Lars G{\aa}rding, \emph{An inequality for hyperbolic polynomials}, Journal of
  Mathematics and Mechanics (1959), 957--965.

\bibitem{gromov2018}
Misha Gromov, \emph{Mean curvature in the light of scalar curvature}, arXiv
  preprint arXiv:1812.09731 (2018).

\bibitem{nelli1995some}
Barbara Nelli and Harold Rosenberg, \emph{Some remarks on embedded
  hypersurfaces in hyperbolic space of constant curvature and spherical
  boundary}, Annals of Global Analysis and Geometry \textbf{13} (1995), no.~1,
  23--30.

\bibitem{reilly1973}
Robert~C. Reilly, \emph{Variational properties of functions of the mean
  curvatures for hypersurfaces in space forms}, Journal of Differential
  Geometry \textbf{8} (1973), no.~3, 465--477.

\bibitem{schoen1975curvature}
Richard Schoen, Leon Simon, and Shing-Tung Yau, \emph{Curvature estimates for
  minimal hypersurfaces}, Acta Mathematica \textbf{134} (1975), no.~1,
  275--288.

\bibitem{simons1968minimal}
James Simons, \emph{Minimal varieties in {R}iemannian manifolds}, Annals of
  Mathematics (1968), 62--105.

\bibitem{souam2019mean}
Rabah Souam, \emph{Mean curvature rigidity of horospheres, hyperspheres and
  hyperplanes}, arXiv preprint arXiv:1912.02669 (2019).

\end{thebibliography}

\end{document}